\documentclass{article}
\usepackage[english]{babel}
\usepackage{geometry,amsmath,amssymb,graphicx,bbm,latexsym,theorem}
\catcode`\<=\active \def<{
\fontencoding{T1}\selectfont\symbol{60}\fontencoding{\encodingdefault}}
\catcode`\>=\active \def>{
\fontencoding{T1}\selectfont\symbol{62}\fontencoding{\encodingdefault}}
\newcommand{\Zeta}{\mathrm{Z}}
\newcommand{\mathd}{\mathrm{d}}
\newcommand{\nin}{\not\in}
\newcommand{\tmaffiliation}[1]{\\ #1}
\newcommand{\tmop}[1]{\ensuremath{\operatorname{#1}}}
\newcommand{\tmtextit}[1]{{\itshape{#1}}}
\newcommand{\tmverbatim}[1]{{\ttfamily{#1}}}
\newenvironment{proof}{\noindent\textbf{Proof\ }}{\hspace*{\fill}$\Box$\medskip}
\newtheorem{conjecture}{Conjecture}
\newtheorem{corollary}{Corollary}
\newtheorem{definition}{Definition}
\newtheorem{lemma}{Lemma}
{\theorembodyfont{\rmfamily}\newtheorem{note}{Note}}
\newtheorem{proposition}{Proposition}
{\theorembodyfont{\rmfamily}\newtheorem{remark}{Remark}}
\newtheorem{theorem}{Theorem}

\begin{document}

\title{A Sequence of Cauchy Sequences Which Is Conjectured to Converge to the
Imaginary Parts of the Zeros of the Riemann Zeta Function}

\author{
  Stephen Crowley <stephencrowley214@gmail.com>
  \tmaffiliation{December 8, 2018}
}

\maketitle

\begin{abstract}
  The convergence of a sequence of Cauchy sequences is conjectured; which if
  shown to be true, would prove the Riemann hypothesis by way of \ LeClair and
  Fran{\c c}a's transcendental equation criteria. 
\end{abstract}

\

{\tableofcontents}

\

\section{Introduction}

LeClair and Fran{\c c}a established criteria for the Riemann hypothesis in
{\cite{z0t}}. Here, a sequence of Cauchy sequences based on complex dynamical
systems involving the Hardy Z function is constructed which explicitly shows
that a solution to $\vartheta (y_n) + S (y_n) = \left( n - \frac{3}{2} \right)
\pi$ should exist for all values of $n$ if it is always possible to choose a
small enough Lipschitz constant.

\subsection{Transcendental Equations Satisifed By The Nontrivial Riemann
Zeros}

\begin{definition}
  The \tmverbatim{exact equation} for the $n$-th zero of the Hardy $Z$
  function $y_n$ is given by {\cite[Equation 20]{z0t}}
  \begin{equation}
    \vartheta (y_n) + S (y_n) = \left( n - \frac{3}{2} \right) \pi \label{ee}
  \end{equation}
  where $y_n$ enumerate the zeros of $Z$ on the real line and the zeros of
  $\zeta$ on the critical line
  \begin{equation}
    \Zeta (y_n) = 0 \tmop{and} \zeta \left( \frac{1}{2} + i y_n \right) = 0
    \forall n \in \mathbbm{Z}^+
  \end{equation}
  where $\mathbbm{Z}^+$ denotes the positive integers. {\cite[Equation
  14]{z0t}} 
\end{definition}

By replacing the $\ln \Gamma$ function in (\ref{vartheta}) with Stirling's
asymptotic expansion as in {\cite[Equation 13]{z0t}} we get
\begin{equation}
  \tilde{\vartheta} (t) = \frac{t}{2} \ln \left( \frac{t}{2 \pi e} \right) -
  \frac{\pi}{8} + O (t^{- 1})
\end{equation}
and substitute $\vartheta (t)$ with $\tilde{\vartheta} (t)$ in Equation
\ref{ee} which leads to

\begin{definition}
  The \tmverbatim{asymptotic equation} for the $n$-th zero of the Hardy $Z$
  function
  \begin{equation}
    \frac{t_n}{2 \pi} \ln \left( \frac{t_n}{2 \pi t} \right) + S (t_n) = n -
    \frac{11}{8} \label{ae}
  \end{equation}
  {\cite[Equation 20]{z0t}}
\end{definition}

\begin{figure}[h]
  \resizebox{6in}{3in}{\includegraphics{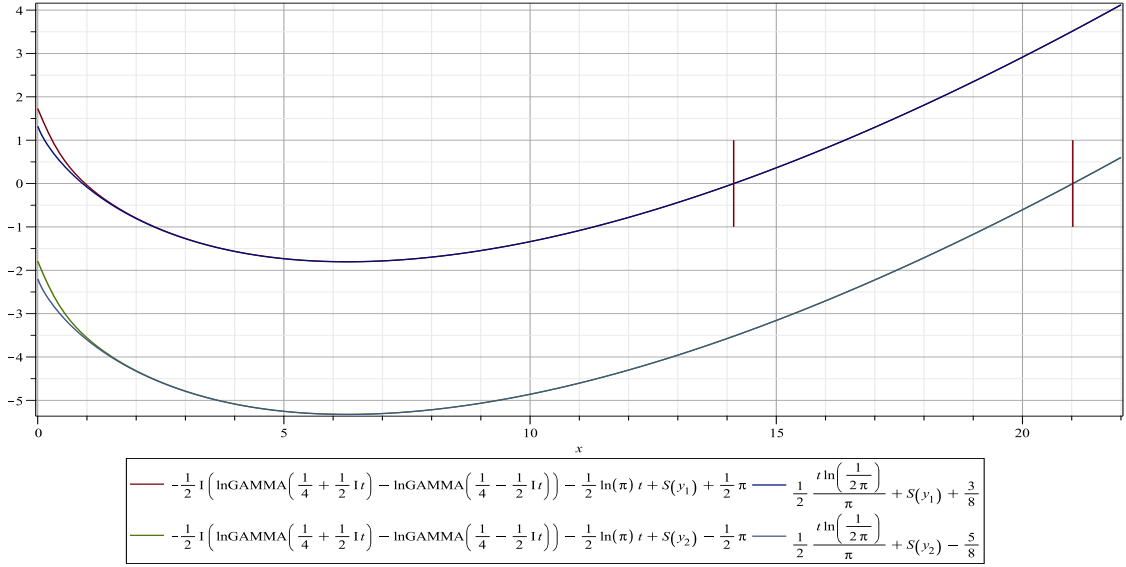}}
  \caption{The functions $\vartheta (y_n) + S (y_n) - \left( n - \frac{3}{2}
  \right) \pi$ and $\tilde{\vartheta} (y_n) + S (y_n) - \left( n - \frac{3}{2}
  \right) \pi$ for $n = 1, 2$ with the zeros at $y_1$ and $y_2$ marked with
  vertical lines. }
\end{figure}

\begin{remark}
  The fact that the exact and asymptotic equations have two solutions when $n
  = 1$ can be understood by noting that Equations (\ref{ee}) and (\ref{ae})
  are derived from the equation
  \begin{equation}
    n = \tilde{\vartheta} (t) + \frac{\pi}{8} - \frac{5}{8} + S (t_{})
    \label{ord}
  \end{equation}
  which has a minimum in the interval $(- 2, - 1)$ and thus $n \geqslant - 1$
  so that, in order to follow the convention that the zeros are enumerated by
  the positive integers, the substituion $n \rightarrow n - 2$ is made in
  Equation (\ref{ord}) so that
  \begin{equation}
    n - 2 = \tilde{\vartheta} (t) + \frac{\pi}{8} - \frac{5}{8} + S (t_{})
  \end{equation}
  {\cite[Equation 12]{z0t}}
\end{remark}

\begin{theorem}
  \label{le}If the limit
  \begin{equation}
    \lim_{\delta \rightarrow 0^+} \arg \left( \zeta \left( \frac{1}{2} +
    \delta + i t \right) \right)
  \end{equation}
  is exists and is well-defined $\forall t$ then the left-hand side of
  Equation (\ref{ae}) is well-defined $\forall t$, and due to monotonicity,
  there must be a unique solution for every $n \in \mathbbm{Z}^+$.
  {\cite[II.A]{z0t}} 
\end{theorem}

\begin{corollary}
  The number of solutions of Equation (\ref{ae}) over the interval $[0, t]$ is
  given by
  \begin{equation}
    N_0 (t) = \frac{t}{2 \pi} \ln \left( \frac{t}{2 \pi e} \right) +
    \frac{7}{8} + S (t) + O (t^{- 1}) \label{N0}
  \end{equation}
\end{corollary}

which counts the number of zeros \tmverbatim{on the critical line}.

\begin{conjecture}
  \label{RH}(The Riemann hypothesis) All solutions $t$ of the equation
  \begin{equation}
    \zeta (t) = 0
  \end{equation}
  besides the trivial solutions $t = - 2 n$ with $n \in \mathbbm{Z}^+$ have
  real-part $\frac{1}{2}$, that is, $\tmop{Re} (t) = \frac{1}{2}$ when $\zeta
  (t) = 0$ and $t \neq - 2 n$.
\end{conjecture}

\begin{definition}
  The Riemann-von-Mangoldt formula makes use of Cauchy's argument principle to
  count the number of zeros \tmverbatim{inside the critical strip} $0 <
  \tmop{Im} (\rho_n) < t$ where $\zeta (\sigma + i \rho_n)$ with $0 < \sigma <
  1$
  \begin{equation}
    N (t) = \frac{t}{2 \pi} \ln \left( \frac{t}{2 \pi e} \right) + \frac{7}{8}
    + S (t) + O (t^{- 1})
  \end{equation}
  and this definition does not depend on the Riemann hypothesis(Conjecture
  \ref{RH}). This equation has exactly the same form as the asymptotic
  Equation \ref{ae}. {\cite[Equation 15]{z0t}}
\end{definition}

\begin{lemma}
  \label{fl}If the exact Equation (\ref{ee}) has a unique solution for each $n
  \in \mathbbm{Z}^+$ then Conjecture \ref{RH}, the Riemann hypothesis,
  follows.
\end{lemma}

\begin{proof}
  If the exact equation has a unique solution for each $n$, then the zeros
  obtained from its solutions on the \tmverbatim{critical line} can be counted
  since they are enumerated by the integer $n$, leading to the counting
  function $N_0 (t)$ in Equation (\ref{N0}). The number of solutions obtained
  on the \tmverbatim{critical line} would saturate counting function of the
  number of solutions on the \tmverbatim{critical strip} so that $N (t) = N_0
  (t)$ and thus all of the non-trivial zeros of $\zeta$ would be enumerated in
  this manner. If there are zeros off of the critical line, or zeros with
  multiplicity $m \geqslant 2$, then the exact Equation (\ref{ee}) would fail
  to capture all the zeros on the critical strip which would mean $N_0 (t) < N
  (t)$. \ {\cite[IX]{z0t}}
\end{proof}

\begin{corollary}
  The Riemann hypothesis(RH) is not necesarily false if the exact Equation
  (\ref{ee}) does not have a unique solution for every $n$, since the
  solutions could still be on the critical line but not necessarily simple,
  that is, a root on the critical line could have multiplicity $m \geqslant 2$
  and the RH would still be true.
\end{corollary}

\begin{corollary}
  The Riemann hypothesis is true and all of the zeros on the critical line are
  simple if the exact Equation (\ref{ee}) has a unique solution for each $n
  \in \mathbbm{Z}^+$. {\cite[IX]{z0t}}
\end{corollary}

\section{Iterated Function Systems}

\subsection{Fixed-Points of Functions}

\begin{definition}
  A fixed-point $\alpha$ of a function $f (x)$ is a value $\alpha$ such that
  \begin{equation}
    f (\alpha) = \alpha
  \end{equation}
  {\cite[3.]{raadstrom1953iteration}}
\end{definition}

\begin{definition}
  The multiplier $\lambda_f (\alpha)$ of a fixed point $\alpha$ of a map $f
  (x)$ is equal to the derivative $\dot{f} (\alpha)$ of the map evaluated at
  the point $\alpha$ which is the first term in the Taylor expansion at that
  point
  \begin{equation}
    \lambda_f (\alpha) = \dot{f} (\alpha)
  \end{equation}
  If $| \lambda_f (\alpha) | < 1$ then $\alpha$ is a said to be an attractive
  fixed-point of $f (x)$. If $| \lambda_f (\alpha) | = 1$ then $\alpha$ is an
  indifferent fixed-point of $f (t)$ also known as as neutral fixed-point, and
  if $| \lambda_f (\alpha) > 1 |$ then $\alpha$ is a repelling fixed-pint of
  $f (t)$. When $| \lambda_f (\alpha) | = 0$ the fixed-point $\alpha$ is said
  to be superattractive fixed-point of $f
  (t)${\cite[3.]{raadstrom1953iteration}}
\end{definition}

\begin{lemma}
  The Banach Fixed-Point Theorem
  
  If $f (x)$ is a continuous function defined on $[a, b]$ and
  \begin{equation}
    f (x) \in [a, b] \forall x \in [a, b]
  \end{equation}
  and there exists some constant $0 < c < 1$ such that
  \begin{equation}
    \text{$\frac{| f (x) - f (y) |}{x - y} \leqslant c$} \label{lc}
  \end{equation}
  then $f (x)$ has a unique fixed-point $x \in [a, b]$ and the sequence $f
  (x_0), f (f (x_0)), f (f (f (x_0))), \ldots$ converges to the unique
  fixed-point of $f (x)$ in the interval $[a, b]$.
\end{lemma}

\subsubsection{An Iteration Function Which Successively Removes Roots}

\begin{definition}
  Let
  \[ Y_{n, m} (t) = \left\{ \begin{array}{ll}
       t & m = 0\\
       t + h_{n, m} \cos (\pi n) \tanh \left( \frac{Z (Y_{n, m - 1} (t))}{|
       \Omega (t) | \prod_{k = 1}^{n - 1} \tanh (Y_{n, m - 1} (t) - y_k)}
       \right) & m \geqslant 1
     \end{array} \right. \]
  denote the $m$-th iterate of the $n$-th iteration function corresponding to
  the $n$-th zero of the Hardy $Z$ function where
  \begin{equation}
    \Omega (t) = \left\{ \begin{array}{ll}
      1 & t = e\\
      e^{\frac{3}{4} \sqrt{\frac{\log (t)}{\log (\log (t))}}} & t \neq e
    \end{array} \right.
  \end{equation}
  is a lower bound for the running maximum of $| Z (s) |$
  \begin{equation}
    \max_{0 \leqslant s \leqslant t} | Z (s) | > \Omega (t) \forall t
    \geqslant 45.590 \ldots
  \end{equation}
  ensuring that
  \begin{equation}
    \frac{| Z (t) |}{\Omega (t)} > 0 \forall t \geqslant 45.590 \ldots
  \end{equation}
  which normalizes the range of $Z (t)$ which is known to grow in both maximum
  and average value as $t \rightarrow \infty$ and $h_{n, m}$ is factor which
  influences the rate of convergence
  \begin{equation}
    h_{n, m} = \left\{ \begin{array}{ll}
      1 & m \leqslant 2\\
      h_{n, m - 1} & \tmop{sign} (\Delta Y^{}_{n, m - 2} (t)) = \tmop{sign}
      (\Delta Y^{}_{n, m - 1} (t))\\
      \frac{h_{n, m - 1}}{2} & \tmop{sign} (\Delta Y^{}_{n, m - 2} (t)) \neq
      \tmop{sign} (\Delta Y^{}_{n, m - 1} (t))
    \end{array} \right.
  \end{equation}
  where
  \begin{equation}
    \Delta Y_{n, m} (t) = Y_{n, m} (t) - Y_{n, m - 1} (t)
  \end{equation}
  is the $1$-st difference of the $m$-th iterate for the $n$-th zero.
  {\cite[Theorem 3.2.3]{ramachandra1995lectures}} 
\end{definition}

\begin{lemma}
  The roots of $Z (t)$ are fixed-points of $Y_{n, m} (t) \forall n, m \in
  \mathbbm{Z}^+$.
\end{lemma}

\begin{proof}
  If $Z (t) = 0$ then $\tanh \left( \frac{Z (t)}{| \Omega (t) | \prod_{k =
  1}^{n - 1} \tanh (t - y_k)} \right) = \tanh \left( \frac{0}{| \Omega (t) |
  \prod_{k = 1}^{n - 1} \tanh (t - y_k)} \right) = \tanh (0) = 0$ so that $Y_n
  (t) = t + \cos (\pi n) 0 = t + 0 = t$ when $Z (t) = 0$.
\end{proof}

\subsubsection{Indifferent Fixed-Points}

\begin{theorem}
  $Y_{n, m} (t)$ has indifferent fixed-points at each point $y_k$ where $k = 1
  \ldots n - 1$
\end{theorem}

\begin{proof}
  The product in the denominator $\prod_{k = 1}^{n - 1} \tanh (t - y_k)
  \rightarrow 0$ smoothly as $t$ approaches any $y_k \in \bigcup_{k = 1}^{n -
  1} y_k$ since $\tanh (0) = 0$ and $\tanh$ is a smooth function. When any
  element of the product is zero the value of the product is zero regardless
  of the values of any other elements of the product. Since $\frac{1}{s}
  \rightarrow \infty$ as $s \rightarrow 0$ and $\tanh (| x |) \rightarrow 1$
  as $| x | \rightarrow \infty$ we have $\tanh (\infty) = 1$ and $\tanh (-
  \infty) = - 1$ so that $Y_n (t) = t + \cos (\pi n) \forall t \in \bigcup_{k
  = 1}^{n - 1} y_k$. Since $Y_n (t) = t \pm 1 \forall t \in \bigcup_{k = 1}^{n
  - 1} y_k$ when $n$ is an integer, we see that $\frac{\mathd}{\mathd t} Y_n
  (t) = \frac{\mathd}{\mathd t} (t \pm 1) = 1$ so that the multiplier
  $\lambda_{Y_n (t)} = \left| \frac{\mathd}{\mathd t} Y_n (t) \right| = 1
  \forall t \in \bigcup_{k = 1}^{n - 1} y_k$.
\end{proof}

\begin{theorem}
  $Y_{n, m} (t)$ has indifferent fixed points at each trivial zero $-
  \frac{i}{2} (- 4 n - 1)$ where $Z \left( - \frac{i}{2} (- 4 n - 1) \right) =
  0 \forall n \in \mathbbm{Z}^{+_{}}$.
\end{theorem}

\begin{proof}
  Since $\frac{\mathd}{\mathd t} (f (t) + g (t)) = \dot{f} (t) g (t) + f (t)
  \dot{g} (t)$ and $Z (t) = e^{i \vartheta (t)} \zeta \left( \frac{1}{2} + i t
  \right)$ it suffices to show that $\lim_{t \rightarrow - \frac{i}{2} (- 4 n
  - 1)} \left| \frac{\mathd}{\mathd t} e^{i \vartheta (t)} \right| = \infty$.
  Since $\frac{\mathd}{\mathd t} e^{i \vartheta (t)} = i \left( \frac{\frac{i
  \Psi \left( \frac{1}{4} - \frac{i t}{2} \right)}{2} - \frac{i \Psi \left(
  \frac{1}{4} + \frac{i t}{2} \right)}{2}}{2} - \frac{\ln (\pi)}{2} \right)
  e^{i \vartheta (t)}$ we only have to check that $\lim_{t \rightarrow
  \frac{i}{2} (- 4 n - 1)} \left| \Psi \left( \frac{1}{4} + \frac{i t}{2}
  \right) \right| = \lim_{t \rightarrow - \frac{i}{2} (- 4 n - 1)} \left| \Psi
  \left( \frac{1}{4} + \frac{i t}{2} \right) \right| = \infty$ which is true
  since $\Psi (t)$ has poles at $t = 1 - n \forall n \in \mathbbm{Z}^+$ where
  $\Psi (t) = \frac{\dot{\Gamma} (t)}{\Gamma (t)}$ and $\Gamma (t)$ has poles
  at $t = 1 - n \forall n \in \mathbbm{Z}^+$. Since $| \tanh (\infty) | = 1$
  the multiplier is equal to 1 at each $- \frac{i}{2} (- 4 n - 1)$.
\end{proof}

\subsubsection{Alternating Attractive and Repulsive Fixed-Points}

\begin{proposition}
  \label{podd}When $n$ is an odd number, $Y_n (t)$ has attractive fixed-points
  at the odd-numbered roots $y_{2 k - 1} \forall 2 k - 1 \geqslant n$ and
  repulsive fixed-points at the even-numbered roots $y_{2 k} \forall 2 k
  \geqslant n$.
\end{proposition}

\begin{proposition}
  \label{peven}When $n$ is an even number, $Y_n (t)$ has attractive
  fixed-points at the even-numbered roots $y_{2 k} \forall 2 k \geqslant n$
  and repulsive fixed-points at the odd-numbered roots $y_{2 k - 1} \forall 2
  k - 1 \geqslant n$.
\end{proposition}

\begin{figure}[h]
  \resizebox{6.5in}{2.5in}{\includegraphics{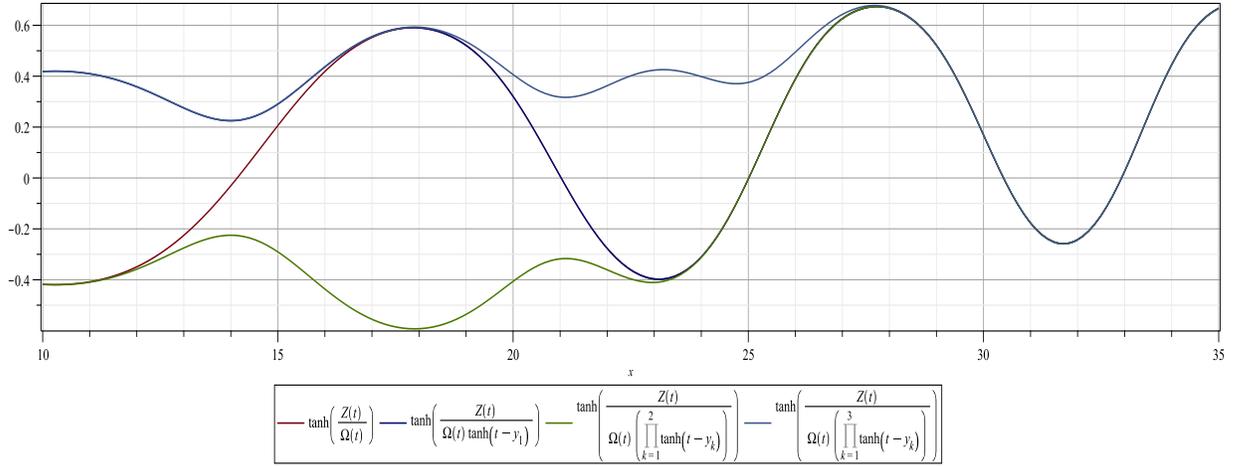}}
  \caption{$\tmop{The}$ functions which are subtracted or added to $t$ to get
  $Y_1 (t), Y_2 (t), Y_3 (t), Y_4 (t)$. When $n$ is odd $\cos (\pi n) = - 1$
  so that the the value is subtracted from $t$, when $n$ is even $\cos (2 \pi)
  = 1$ so it is added. It is plain to see that the curves $\tanh \left(
  \frac{Z (t)}{\Omega (t) \prod_{k = 1}^{n - 1} \tanh (t - y_k)} \right)$ do
  not cross the zero axis for any $t < y_n$ }
\end{figure}

\

\begin{figure}[h]
  \resizebox{6.5in}{2in}{\includegraphics{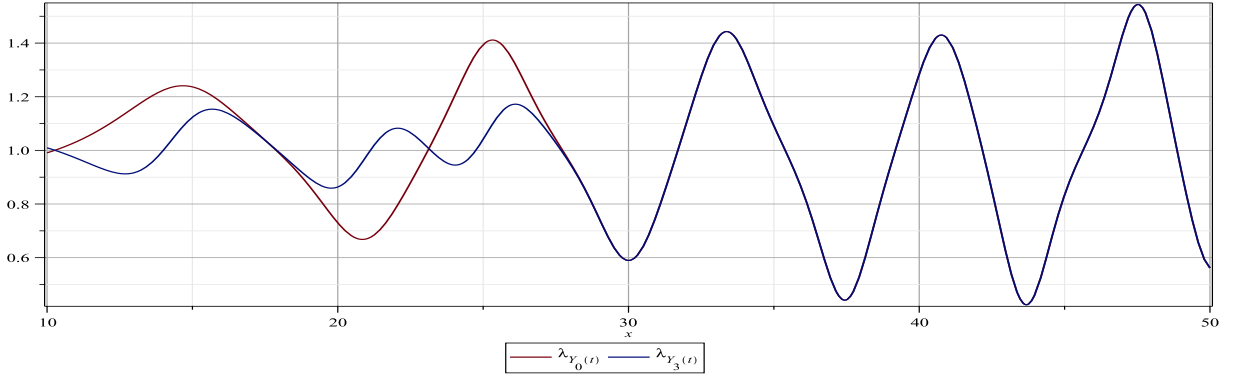}}
  \caption{Multipler of the maps $Y_1 (t)$ and $Y_3 (t)$}
\end{figure}

\begin{remark}
  The function $h_{n, m}$ is defined to be $1$ when $1 \leqslant m \leqslant
  2$. If $\tmop{sign} (\Delta Y^{}_{n, m - 1} (t)) \neq \tmop{sign} (\Delta
  Y^{}_{n, m} (t))$ then $h_{n, m + 1} = \frac{h_{n, m}}{2}$ so that the
  convergence rate is \ halved when the sign of the difference between
  successive iterates changes, indicating that it jumped across the root. This
  prevents the sequence generated by the iteration from getting stuck in an
  artifical $2$-cycle and jumping back and forth across the root with equal
  magnitude indefinately when implementing this method with finite-precision
  arithmetic on a digital computer. Without this successive relaxation, the
  iterates still converge in theory however the number of iterations required
  could be several million or higher, while still having the difficulty of
  possibly getting stuck in a $2$-cycle in computer implementations.
\end{remark}

\subsection{Contraction Mappings}

\begin{theorem}
  \label{cm}The Lipschitz constant $M$ of the map $Y_{n, m} (t) < 1 \forall t
  > e$ therefore $Y_{n, m} (t)$ is a contraction mapping
  \begin{equation}
    | Y_{n, m} (t) - Y_{n, m} (s) | \leqslant M | t - s |
  \end{equation}
\end{theorem}

\begin{proof}
  The Lipschitz constant of a continuous differentiable function $f (x)$ is
  equal to the maximum absolute value of its derivative
  \begin{equation}
    M = \sup_x \left| \frac{\mathd}{\mathd x} f (x) \right|
  \end{equation}
  The derivative of $t - \tanh (t)$ is $\tanh (t)^2$. Since the maximum
  absolute value of $\tanh (t)$ is 1 then the maximum value of its square is
  also $1$. Since $\Omega (t) > 1$ and $h_{n, m} \leqslant 1$ the derivative
  $\frac{\mathd}{\mathd t} Y_{n, m} (t)$ can never have an absolute value
  $\geqslant 1$ since that would require $\left| \tanh \left( \frac{Z (t)}{|
  \Omega (t) | \prod_{k = 1}^{n - 1} \tanh (t - y_k)} \right) \right| = 1$
  which is only possible if $Z (t) = \pm \infty$ which is only the case when
  $t = \pm \frac{i}{2}$ which corresponds to the pole at $\zeta (1)$. Since $Z
  (t) \in \mathbbm{R}$ when $t \in \mathbbm{R}$ it can never be the case that
  $Z (t) = \infty$ so that $\left| \frac{\mathd}{\mathd t} Y_{n, m} (t)
  \right| \neq 1 \forall t \in \mathbbm{R}$ and the Lipschitz constant $M$ is
  strictly less than 1.
\end{proof}

\subsubsection{Sequential Convergence to the Nearest Fixed-Points}

\begin{proposition}
  \label{p1}The limit
  \begin{equation}
    y_n = \lim_{m \rightarrow \infty} Y_{n, m} (s_n)
  \end{equation}
  where
  \begin{equation}
    s_n = \left\{ \begin{array}{ll}
      14 & n = 1\\
      21 & n = 2\\
      \frac{y_{n - 1} + y_{n - 2}}{2} & n \geqslant 3
    \end{array} \right.
  \end{equation}
  exists and is equal to the $n$-th zero of the Hardy Z function for all
  integer $n \subset \mathbbm{Z}^+$. That is, $Y_{n, m} (z_n)$ forms a Cauchy
  sequence, due to the contraction mapping property proved in Theorem \ref{cm}
  whose elements are indexed by $m$ converging to the $n$-th root $y_n$ where
  the $n$-th starting point is defined to be half-way between the $(n - 2)$-th
  and the $(n - 1)$-th root $y_n$ when $n > 2$ and equal to a point close to
  the first known zero at $14.134 \ldots .$ when $n = 1$ and a point close to
  the $2 \tmop{nd}$ zero at $21.022 \ldots$ when $n = 2$
\end{proposition}

\begin{remark}
  The mid-way point between the nearest neighbors to the left of $y_n$ is used
  as the starting point for the iteration since any point less than $y_n$ and
  greater than $e$ is within the immediate basin of attraction of $y_n$. The
  precise location of any roots $y_p$ where $p < n$ cannot be used as a
  starting point since the map $Y_{p, m} (t)$ is a non-expansive mapping with
  Lipschitz constant precisely equal to 1 when $t \in \bigcup^{p - 1}_{k = 1}
  y_k$ so that the hyperbolic tangent has an argument of infinity resulting in
  a value of 1. Trajectories are neither attracted or repelled to any point
  $\bigcup^{n - 1}_{k = 1} y_k$ under the action of the map $Y_{n, m} (t)$
  however, trajectories started precisely on any point $t \in \bigcup^{n -
  1}_{k = 1} y_k$ will never attain a value other than $t$ since any $y_k$ is
  a fixed-point of $Y_{n, m} (t)$. 
\end{remark}

\begin{note}
  The truth of Propositon \ref{p1} has been verified computationally up to $n
  = 800, 000$ with a computer program which implements the methods described
  here using the arbitrary precision complex ball arithmiticlibrary
  arblib{\cite{Johansson2017arb}} and compares the results against the tables
  published by Andrew Odlyzko{\cite{zt}}. 
\end{note}

\begin{theorem}
  The Cauchy sequence $\lim_{m \rightarrow \infty} Y_{n, m} (s_n)$ will never
  converge to any $y_k$ where $k < n$. 
\end{theorem}

\begin{proof}
  All $y_k$ are indifferent fixed-points of $Y_{n, m} (t)$ and the
  trajectories generated by $Y_{n, m} (s_n)$ are never started from a point
  $y_k$ since $s_n \nin \bigcup_{k = 1}^{n - 1} y_k$ and the only way $Y_{n,
  m} (t)$ would $'' \tmop{convege}''$ to an indifferent fixed-point is if it
  was started precisely on one, and $s_n$ is by definition equal to the
  mid-point between successive $y_n$.
\end{proof}

\begin{theorem}
  \label{pc}The Cauchy sequence $Y_{n, m} (s_n)$ will never converge to any
  $y_{n + 2 k - 1} \forall k \in \mathbbm{Z}^+$ if Proposition \ref{podd} is
  true.
\end{theorem}

\begin{proof}
  If Propositions \ref{podd} is true then $y_{n + 2 k - 1}$ are repelling
  fixed-points for $Y_{n, m} (t)$.
\end{proof}

\begin{note}
  If Propositions \ref{podd} and \ref{peven} are true then $Y_{n, m} (s_n)$
  will never converge to $y_q$ with $q$ odd and $n$ even nor to $y_r$ with $r$
  even and $n$ odd. It suffices to prove that $Y_{n, m} (s_n) < y_{n + 1}
  \forall n, m \in \mathbbm{Z}^+$ which would mean that $Y_{n, m} (s_n)$ can
  never jump across the repelling fixed-point at $y_{n + 1}$ to land on any of
  the attractive fixed-points in $\bigcup_{k = 1}^{\infty} y_{n + 2 k}$
\end{note}

\begin{lemma}
  Let
  \begin{equation}
    Y^+_{n, m} (t) = \left\{ \begin{array}{ll}
      t & m = 0\\
      t + h_{n, m} \cos (\pi n) \tanh \left( \frac{Z (Y_{n, m - 1} (t))}{|
      \Omega (t) | \prod_{k = 1}^{n - 1} \tanh (Y_{n, m - 1} (t) - y_k)}
      \right) & m \geqslant 1
    \end{array} \right.
  \end{equation}
  \begin{equation}
    Y^-_{n, m} (t) = \left\{ \begin{array}{ll}
      t & m = 0\\
      t - h_{n, m} \cos (\pi n) \tanh \left( \frac{Z (Y_{n, m - 1} (t))}{|
      \Omega (t) | \prod_{k = 1}^{n - 1} \tanh (Y_{n, m - 1} (t) - y_k)}
      \right) & m \geqslant 1
    \end{array} \right.
  \end{equation}
  and
  \begin{equation}
    z_n = \min (\lim_{m \rightarrow \infty} Y^+_{n, m} (t), \lim_{m
    \rightarrow \infty} Y^-_{n, m} (t))
  \end{equation}
  which must exist because there is known to be an infinity of zeros on the
  critical line. 
\end{lemma}

\begin{proof}
  The only way $z_n$ would not exist is if all the roots $y_k$ were
  indifferent fixed-points$\forall k > n$ but that is impossible since there
  are no indifferent fixed-points of $Y_{n, m} (t)$ because for a fixed-point
  $y_k$ to be indifferent would require $\tanh \left( \frac{Z (Y_{n, m - 1}
  (t))}{| \Omega (t) |} \right) = 1$ which is only possible if $| Z (Y_{n, m -
  1} (t)) | = \infty$ for some $m \in \mathbbm{Z}^+$ and the $Z$ function only
  takes on the value $\infty$ when $t = - \frac{i}{2}$ which corresponds to
  the pole at $\zeta (1)$ since $\frac{1}{2} + i \left( - \frac{i}{2} \right)
  = 1$.
\end{proof}

\begin{definition}
  The \tmverbatim{multiplicity} $m_f (t)$ of a root $\alpha$ is a root $f
  (\alpha) = 0$ such that its Taylor expansion about the point $\alpha$ has
  the form
  \begin{equation}
    f (t) = c (t - \alpha)^{m_f (t)} + (\tmop{higher} \tmop{order}
    \tmop{terms} \ldots)
  \end{equation}
  where $c \neq 0$ and $m \geqslant 1$. The multiplicity of a root $t$ is
  related to the multipler $\lambda_f (t)$ through the formula
  \begin{equation}
    m_f (t) = \frac{1}{1 - \lambda_{N_f} (t)}
  \end{equation}
\end{definition}

where
\begin{equation}
  \lambda_{N_f} (t) = \frac{f (t) \ddot{f} (t)}{\dot{f} (t)}
\end{equation}
is the first derivative of the Newton map of $f (t)$
\begin{equation}
  N_f (t) = t - \frac{f (t)}{\dot{f} (t)}
\end{equation}

\begin{lemma}
  (Milnor's Lemma) Every \tmverbatim{simple root} of $f (t)$ is a
  super-attractive fixed-point of $N_f (t)$ since a superattractive
  fixed-point is one such that its multiplier $\lambda_f (t) = 0$ so that its
  multiplicity is
  \begin{equation}
    m_f (t) = \frac{1}{1 - \lambda_{N_f} (t)} = \frac{1}{1 - 0} = \frac{1}{1}
    = 1
  \end{equation}
  See {\cite[p.52]{milnor2006dynamics}}
\end{lemma}

\begin{proof}
  Let $\alpha$ be a root $Z (\alpha) = 0$ then the multiplier of its Newton
  map is $\lambda_{N_Z} (\alpha) = \lambda_f (\alpha) = \frac{Z (\alpha)
  \ddot{Z} (\alpha)}{\dot{Z} (\alpha)} = 0$ since $Z (\alpha) = 0$ the entire
  expression $\frac{Z (\alpha) \ddot{Z} (\alpha)}{\dot{Z} (\alpha)}$ is equal
  to 0 since due to the ordering of operations the value of $\dot{Z} (t)$ or
  $\ddot{Z} (t)$ is never required to be known in order to know the value of
  $\lambda_{N_Z} (t)$ when $Z (t) = 0$.\quad If any term in the product is 0
  then the entire product takes the value 0. The multiplicity is related to
  the multiplier by $m_Z (t) = \frac{1}{1 - \lambda_{N_Z} (t)} = 1$ and
  therefore simple. Since \ $m_Z (t) = \frac{1}{1 - \lambda_{N_Z} (t)} \forall
  \lambda_{N_Z} (t) \neq 1$then it is known that $\lambda_{N_Z} (t) = 0$ when
  $Z (t_{}) = 0$ \ therefore the point $\alpha$ is a superattractive
  fixed-point corresponding to a simple zero at $\alpha$. Since we now know
  that $m_Z = (\alpha)$ and therefore the zero at $Z (\alpha) = 0$ is simple,
  we therefore know that the denominator $\dot{\Zeta} (t)$ of the multiplier
  $\lambda_{N_Z} (t)$ cannot vanish so that \ \ $\dot{\Zeta} (\alpha) \neq 0$
  since that would imply that $\alpha$ is not a simple root, which would be a
  contradiction to the already established fact that $m_Z (\alpha) = 1$ when
  $Z (\alpha) = 0$.
\end{proof}

\begin{conjecture}
  The roots generated by the sequence $y_n = \lim_{m \rightarrow \infty} Y_{n,
  m} (t)$ are simple
\end{conjecture}

\begin{conjecture}
  Let
  \begin{equation}
    c_n (\varepsilon) = \frac{Z (\max_{t \in [0, y_n]} \{ Y_{n + 1, 1} (t)
    \geqslant t \} + \epsilon) - Z (\min_{t \in [y_n, \infty]} \{ Y_{n + 1, 1}
    (t) \leqslant t \} - \epsilon)}{2 \varepsilon + \max_{t \in [0, y_n]} \{
    Y_{n + 1, 1} (t) \geqslant t \} - \min_{t \in [y_n, \infty]} \{ Y_{n + 1,
    1} (t) \leqslant t \}}
  \end{equation}
  denote the Lipschitz constant in Formula \ref{lc} then it is always possible
  to choose a small enough positive $\varepsilon$ such that $0 < c_n
  (\varepsilon) < 1$.
\end{conjecture}

\section{Appendix}

\subsection{Definitions}

Let $\zeta (t)$ be the Riemann zeta function
\begin{equation}
  \begin{array}{lll}
    \zeta (t) & = \sum_{n = 1}^{\infty} n^{- s} & \forall \tmop{Re} (s) > 1\\
    & = (1 - 2^{1 - s}) \sum_{n = 1}^{\infty} n^{- s} (- 1)^{n - 1} & \forall
    \tmop{Re} (s) > 0
  \end{array}
\end{equation}
and $\vartheta (t)$ be Riemann-Siegel vartheta function

\begin{equation}
  \vartheta (t) = - \frac{i}{2} \left( \ln \Gamma \left( \frac{1}{4} + \frac{i
  t}{2} \right) - \ln \Gamma \left( \frac{1}{4} - \frac{i t}{2} \right)
  \right) - \frac{\ln (\pi) t}{2} \label{vartheta}
\end{equation}
so that the Hardy $Z$ function{\cite{HardyZ}} can be defined by
\begin{equation}
  \begin{array}{ll}
    Z (t) & = e^{i \vartheta (t)} \zeta \left( \frac{1}{2} + i t \right)
  \end{array} \label{Z}
\end{equation}
which is real-valued when $t$ is real and satisfies the identity
\begin{equation}
  \zeta (t) = e^{- i \vartheta \left( \frac{i}{2} - i t \right)} Z \left(
  \frac{i}{2} - i t \right) \label{Zz}
\end{equation}

where $\ln \Gamma (z)$ is the principal branch of the logarithm of the
$\Gamma$ function defined by
\begin{equation}
  \ln \Gamma (z) = \ln (\Gamma (z)) = (z - 1) ! = \prod_{k = 1}^{z - 1} k
  \forall z \in \mathbbm{R}> 0
\end{equation}
which is analytically continued from the positive real axis when $z \in
\mathbbm{C}$ is complex. Each of the points $z \in \mathbbm{Z}= \{ 0, - 1, -
2, \ldots \}$ is a singularity and a branch point so that the union of the
branch cuts is the negative real axis. On the branch cuts, the values of $\ln
\Gamma (z)$ are determined by continuity from above. Let $S (t)$ denote the
normalized argument of $\zeta (t)$ on the critical line
\begin{equation}
  \begin{array}{ll}
    S (t) & = \pi^{- 1} \arg \left( \zeta \left( \frac{1}{2} + i t \right)
    \right)\\
    & = - \frac{i}{2 \pi} \left( \ln \zeta \left( \frac{1}{2} + i t \right) -
    \ln \zeta \left( \frac{1}{2} - i t \right) \right)\\
    & = \frac{1}{\pi} \lim_{\varepsilon \rightarrow 0} \tmop{Im} \left( \ln
    \zeta \left( \frac{1}{2} + i t + \varepsilon \right) \right)
  \end{array} \label{S}
\end{equation}
\begin{definition}
  The \tmverbatim{critical line} is the line in the complex plane defined by
  $\tmop{Re} (t) = \frac{1}{2}$.
\end{definition}

\begin{definition}
  The \tmverbatim{critical strip} is the strip in the complex plane defined by
  $0 < \tmop{Re} (t) < 1$.
\end{definition}

\end{document}